\documentclass[12pt]{article}
\usepackage{booktabs}
\usepackage{caption}
\usepackage{mathrsfs}
\usepackage{amsmath}
\usepackage{amsfonts,amsthm,amssymb,mathrsfs,bbding}
\usepackage{graphics,multicol}
\usepackage{graphicx}
\usepackage{color}
\usepackage{float}
\usepackage{subfigure}
\usepackage{enumerate}
\usepackage{caption}
\usepackage{rotating}
\usepackage{lscape}
\usepackage{longtable}
\allowdisplaybreaks[4]
\usepackage{tabularx}
\usepackage[colorlinks=true,anchorcolor=blue,filecolor=blue,linkcolor=blue,urlcolor=blue,citecolor=blue]{hyperref}
\usepackage{extarrows}
\usepackage{cite}
\usepackage{tikz}
\usepackage{latexsym,bm}
\usepackage{mathtools}
\evensidemargin=\oddsidemargin\topmargin=-1.5cm

\newtheorem{thm}{Theorem}

\newtheorem{lem}[thm]{Lemma}
\newtheorem{cor}[thm]{Corollary}

\newtheorem{claim}{Claim}
\theoremstyle{definition}

\addtocounter{section}{0}
\begin{document}

	\title{\bf The maximum number of edges in $C_{2k+1}^{-}$-free unbalanced signed graphs with bounded clique number \footnote{This project is supported by the National Natural Science Foundation of China (No. 12331012) and Hunan Provincial Innovation Foundation for Postgraduate (No. CX20250747).}}

	\author{{Junjie Wang\footnote{\emph{E-mail address:} wangjunjie2827@163.com}, Yaoping Hou\footnote{Corresponding author.}\setcounter{footnote}{-1}\footnote{\emph{E-mail address:} yphou@hunnu.edu.cn}}\\[2mm]
\small College of Mathematics and Statistics, Hunan Normal University,\\
		\small Changsha,  Hunan, 410081, China\\[2mm]
		}

	\date{}
	\maketitle
	{\flushleft\large\bf Abstract } Recently, the Tur\'{a}n problem for graphs with bounded clique number has attracted considerable attention. Since a graph can be regarded as a signed graph without negative edges, it is natural to extend the study of such Tur\'{a}n-type problems to signed graphs. With this motivation, we investigate the Tur\'{a}n-type problem for unbalanced $C_{2k+1}^{-}$-free signed graphs with bounded clique number. In fact, we establish a general result via the classical stability theorem. Specifically, for sufficiently large $n$ and a color-critical graph $F$, we determine the maximum number of edges among all $n$-vertex $C_{2k+1}^{-}$-free unbalanced signed graphs whose underlying graphs are $F$-free.
    
	\begin{flushleft}

\textbf{AMS classification}: 05C22; 05C35\\

		\textbf{Keywords:} Tur\'{a}n problem; signed graph; odd cycle; color-critical graph.
	\end{flushleft}

	\section{Introduction}

    The Tur\'{a}n problem is one of the central topics in extremal graph theory, which is concerned with the maximum number of edges in an $n$-vertex graph that does not contain $F$ as a subgraph. A graph $G$ is called $\mathcal{F}$-free if $G$ does not contain any member of the family $\mathcal{F}$ of graphs as a subgraph. The \emph{Tur\'{a}n number} $ex(n,\mathcal{F})$ is the maximum number of edges in an $n$-vertex $\mathcal{F}$-free graph. If $\mathcal{F}$ contains one graph $F$, then we write $ex(n,F)$ instead of $ex(n,\mathcal{F})$. 
    
    The well-known Tur\'{a}n theorem states that if a graph contains no complete subgraph $K_{r+1}$, then the maximum number of edges it can contain is attained by the \emph{Tur\'{a}n graph} $T_{r}(n)$, that is, a complete balanced $r$-partite graph on $n$ vertices. Recently, the Tur\'{a}n problem for graphs with bounded clique number has attracted considerable attention. Let \(\mathrm{ex}(n, r, H)\) be the maximum number of edges in an \(n\)-vertex \(\mathcal{F}\)-free graph whose clique number is at most \(r\). In 2024, Alon and Frankl \cite{AF24} determined the exact value of \(\mathrm{ex}(n,r,M_{s+1})\), where \(M_{s+1}\) denotes a matching consisting of \(s+1\) independent edges. Moreover, they investigated the case where the clique number condition is replaced by forbidding a color-critical graph \(F\), and determined the exact value of \(\mathrm{ex}(n,\{F,M_{s+1}\})\). Subsequently, Katona and Xiao \cite{KX24} determined the exact value of $\mathrm{ex}(n,r,P_{k})$ when $k>2r+1$. Furthermore, they conjectured that if the clique number condition is replaced by forbidding a graph $H$ with $\chi(H)\ge 3$, then $\mathrm{ex}(n,\{H,P_{k}\})=n\max\left\{\left\lfloor k/2 \right\rfloor-1, \mathrm{ex}(k-1,H) / (k-1)\right\}+O_{k}(1)$. This conjecture was later resolved by Liu and Kang \cite{LK24}. In 2025, Lu, Liu and Kang \cite {LLK25} determined the exact value of $\mathrm{ex}(n, r, tS_{\ell})$ for sufficiently large $n$, where $tS_{\ell}$ denotes the disjoint union of $t$ copies of the star $S_{\ell}$. The study of cycles plays a pivotal role in extremal graph theory. Let $\mathcal{C}_{\ge k}$ denote the set of cycles of length at least $k$. Dou, Ning and Peng investigated the value of $\mathrm{ex}(n, r, \mathcal{C}_{\ge k})$ and proved Katona and Xiao's conjecture in a stronger form. Many other interesting extremal problems are studied as well, see \cite{Ger25,LZL25,XK25}.

    Since a graph can be interpreted as a signed graph without negative edges, the properties of a graph can also be considered in terms of a signed graph. Therefore, it is natural to consider Tur\'{a}n-type problems for signed graphs. We first give some relevant concepts of signed graphs below. Let \( G \) be a graph with vertex set \( V = V(G) \) and edge set \( E = E(G) \). Then, a \emph{signed graph} is defined as \(\dot{G}=(G, \sigma)\), where \(G\) is the underlying graph and \(\sigma:E\rightarrow\{+1,-1\}\) is the sign function. If all edges get signs $+1$(resp. $-1$), then $\dot{G}$ is called \emph{all positive}(resp. \emph{all negative}), denoted by $(G,+)$ (resp. $(G,-)$). The sign of a cycle \(\dot{C}\)(resp. path $\dot{P}$) is defined by \(\sigma(\dot{C}) = \prod_{e \in E(C)} \sigma(e)\)(resp. \(\sigma(\dot{P}) = \prod_{e \in E(P)} \sigma(e)\)). If \(\sigma(\dot{C}) = +1\) (resp. \(\sigma(\dot{C}) = -1\)), then \(\dot{C}\) is called \emph{positive} (resp. \emph{negative}). A positive (resp. negative) cycle of length \(k\) is denoted by \(C_k^+\) (resp. \(C_k^-\)). An important feature of signed graphs is the concept of switching the signature. For a subset $U$ of $V(\dot{G})$, let $\dot{G}_U$ denote the signed graph obtained from $\dot{G}$ by reversing the sign of each edge between $U$ and $V(\dot{G})\setminus U$. Then we say that $\dot{G}$ and $\dot{G}_U$ are \emph{switching equivalent}. It is not difficult to see that each cycle in $\dot{G}$ maintains its sign after a switching. Hence, $\dot{G}'$ and $\dot{G}_{U}$ have the same set of negative cycles. A signed graph is called \emph{balanced} if it is switching equivalent to an all positive signed graph; otherwise, it is called \emph{unbalanced}. Thus, a signed graph is balanced if and only if it contains no negative cycle. Two signed graphs $\dot{G} = (G, \sigma)$ and $\dot{G}' = (G', \sigma')$ are \emph{switching isomorphic} if $\dot{G}'$ is isomorphic to $\dot{G}_{U}$ for a subset $U$ of $V(G)$. A signed graph $\dot{G}$ is $\dot{H}$-free if it has no subgraph switching isomorphic to $\dot{H}$.

    The Tur\'{a}n-type problem of signed graphs has been studied in recent years. For example, Wang, Hou, and Li \cite{WHL24} determined the maximum number of edges in $C_{3}^{-}$-free unbalanced signed graphs. This result was later extended by Wang et al. \cite{WHH25} to the case of $C_{2k+1}^{-}$-free unbalanced signed graphs. A signed graph version of Tur\'{a}n theorem was studied by Xiong and Hou \cite{XH24}. They determined the maximum number of edges in $\mathcal{K}_{r+1}^-$-free unbalanced signed graphs, where $\mathcal{K}_{r+1}^-$ is the set of all unbalanced signed complete graphs on $r + 1$ vertices. For other interesting extremal problems on signed graphs, see \cite{BCKW18, Brun22, CY24, WL25}.

    In this paper, motivated by the study of classical extremal graph theory, we focus on analogues of the Tur\'{a}n-type problem for signed graphs with bounded clique number. In fact, we establish a general result. Let \(C_{3}^{-} \cdot T_{r}(n-2)\) denote the signed graph obtained by identifying a vertex of \(C_{3}^{-}\) with a vertex of \((T_{r}(n-2), +)\). Let $\mathcal{H}$ be the family of graphs obtained as follows. Start with a complete tripartite graph $K$ with partite sets $S_1,S_2,S_3$, where
$$
(|S_1|,|S_2|,|S_3|)
\in
\begin{cases}
\{(k+1,k-2,k-2),(k,k-1,k-2)\}, & n=3k,\\[4pt]
\{(k,k-1,k-1),(k+1,k-1,k-2)\}, & n=3k+1,\\[4pt]
\{(k+1,k-1,k-1)\}, & n=3k+2.
\end{cases}
$$
Then add a triangle $xyz$ disjoint from $K$, and join two vertices of this triangle, say $x$ and $y$, to every vertex of the largest partite set $S_1$. It is easy to verify that the number of edges is the same for every graph in $\mathcal{H}$. Let $\dot{\mathcal{H}}$ denote the corresponding family of signed graphs obtained from the graphs in $\mathcal{H}$ by assigning a negative sign to every edge incident with $x$, except for $xz$, and a positive sign to every remaining edge. A graph is \emph{color-critical} if it contains an edge whose deletion reduces its chromatic number. Our main results are as follows.

\begin{thm}\label{thm::main}
Let $r$ and $k$ be integers with $r\ge 3$ and $k\ge 2$. Let $F$ be a color-critical graph with $\chi(F)=r+1$, and let $G$ be an $n$-vertex $F$-free graph. For sufficiently large $n$, if $\dot{G}=(G,\sigma)$ is an unbalanced $C_{2k+1}^{-}$-free signed graph with the maximum number of edges, then $\dot{G}$ is switching isomorphic to a member of $\dot{\mathcal{H}}$ when $r=3$ and $k=2$, and is switching isomorphic to $C_3^{-}\cdot T_r(n-2)$ otherwise.
\end{thm}

In Theorem \ref{thm::main}, we assume that the chromatic number of the color-critical graph $F$ satisfies $\chi(F) \ge 4$. In fact, the corresponding conclusion holds trivially when $\chi(F)=3$. It should be noted that the unique extremal graph of three color-critical graphs is the graph $T_2(n)$, while odd cycles are also three color-critical graphs. Therefore, up to switching isomorphism, the unbalanced bipartite Tur\'{a}n graph is the unique extremal signed graph. If we replace $C_{2k+1}^{-}$ with $C_{2k}^{-}$ in Theorem \ref{thm::main}, obviously, $(T_{r}(n), -)$ satisfies the condition and has the maximum number of edges. 

Note that the complete graph is a color-critical graph. The following result is an immediate consequence of Theorem \ref{thm::main}. 

\begin{thm}\label{cor::1}
    Let $r$ and $k$ be integers with $r \ge 3$ and $k \ge 2$. Let $\dot{G} = (G, \sigma)$ is an unbalanced $C_{2k+1}^{-}$-free signed graph on $n$ vertices with the clique number at most $r$. For sufficiently large $n$, if $\dot{G}$ is extremal, then $\dot{G}$ is switching isomorphic to a member of $\dot{\mathcal{H}}$ when $r=3$ and $k=2$, and is switching isomorphic to $C_3^{-}\cdot T_r(n-2)$ otherwise.
\end{thm}

Let $\mathcal{\dot{C}}_{\ge k}^{-}$ be the set of negative cycles of length at least $k$. By Theorem  \ref{cor::1}, we can derive an analogue of the result obtained by Dou et al.\cite{DNP25} for signed graphs.

\begin{cor}\label{cor::2}
    Let $r$ and $k$ be integers with $r \ge 3$ and $\ell \ge 4$. Let $\dot{G} = (G, \sigma)$ is a $\mathcal{\dot{C}}_{\ge \ell}^{-}$-free unbalanced signed graph on $n$ vertices with the clique number at most $r$. For sufficiently large $n$, if $\dot{G}$ is extremal, then $\dot{G}$ is switching isomorphic to $C_{3}^{-} \cdot T_{r}(n-2)$.
\end{cor}

	\section{Notation and preliminaries}

     First, we describe notations and terminologies. Let $G$ be a graph with vertex set $V(G)$ and edge set $E(G)$, and let $e(G)=|E(G)|$ denote the size of $G$. For a vertex $v \in V(G)$, the neighborhood $N(v)$ of $v$ is $\{u| uv \in E(G)\}$, and the degree $d(v)$ of $v$ is $|N(v)|$. For $S \subseteq V(G)$ and $v \in V(G)$, let $d_S(v) = |N_S(v)| = |N(v)\cap S|$. We denote by $E_{G}(S,T)$ the set of edges between $S$ and $T$ for any two disjoint subsets $S$ and $T$ of $V(G)$, which is abbreviated as $E(S,T)$. Let $e(S,T)=|E(S,T)|$. For any subset of vertices $S$, let the induced subgraph of $S$ be $G[S]$. For disjoint subsets \(A,B \subseteq V(G)\), we denote \(G[A,B]\) as the induced bipartite subgraph of \(G\) with parts $A$ and $B$. Let $P = x_1x_2\cdots x_m$ be a path in $G$. For $x_i,x_j \in V(P)$, we use $x_iPx_j$ to denote the subpath of $P$ between $x_i$ and $x_j$.

      The following result was obtained by Simonovits \cite{Sim68}. 
      
      \begin{lem}(Simonovits\cite{Sim68})\label{Sthm}
      For any color-critical graph $H$ with $\chi(H) = r + 1 \ge 3$ and for sufficiently large $n$, the Tur\'{a}n graph $T_r(n)$ is the unique graph which attains the maximum number of edges in an $n$-vertex $H$-free graph. 
      \end{lem}

     The following classical stability theorem was proved by Erd\H{o}s and Simonovits \cite{Erd66,Erd68,Sim68}.
     
     \begin{lem}(Erd\H{o}s, Simonovits\cite{Erd66,Erd68,Sim68})\label{ESthm}
         Let $H$ be a graph with $\chi(H) = r + 1 \geq 3$. Then, for every $\varepsilon > 0$, there exist $\beta = \beta(H, \varepsilon) > 0$ and $n_0 = n_0(H, \varepsilon) \in \mathbb{N}$ such that the following holds. If $G$ is an $H$-free graph on $n \geq n_0$ vertices with $e(G) \geq e(T_r(n)) - \beta n^2$, then there exists a partition of $V(G) = V_1 \cup \cdots \cup V_r$ such that $\sum_{i=1}^r e(V_i) < \varepsilon n^2$. Therefore, $G$ can be obtained from $T_r(n)$ by adding and deleting a set of at most $\varepsilon n^2$ edges.
     \end{lem}

     \begin{lem}(Cioab\v{a}, Feng, Tait, Zhang\cite{CFT20})\label{lem::jh}
         Let \( A_1, \cdots, A_p \) be finite sets. Then
         $$
         \left|\bigcap_{i=1}^{p} A_i\right| \geq \sum_{i=1}^p |A_i| - (p - 1)\left| \bigcup_{i=1}^p A_i \right|.
         $$
     \end{lem}

     The following result will be frequently used in our proofs, which was given in \cite{Z82}. 

     \begin{lem}(Zaslavsky\cite{Z82})\label{Z82}
		Let $G$ be a connected graph and $T$ a spanning tree of $G$. Then each
		switching equivalence class of signed graphs on the graph $G$ has a unique representative which is +1 on T. Indeed, given any prescribed sign function $\sigma_{T}$: $T \to \{+1,-1\}$, each switching class has a single representative which agrees with $\sigma_{T}$ on $T$.
	\end{lem}

    The following conclusion is crucial to the proof of Theorem \ref{thm::main}.

     \begin{lem}\label{S_1}
        Let $k$ and $m$ be integers with $k \ge 1$ and $m \ge 3$, and let $K_{n_1,n_2,\cdots,n_m}$ be a complete $m$-partite graph with parts of size $n_1, n_2, \cdots, n_m$, where each part has size at least $2k$.  If $\Gamma = (K_{n_1,n_2,\cdots,n_m}, \sigma)$ is a $C_{2k+1}^{-}$-free signed graph, then $\Gamma$ is balanced. 
    \end{lem}

    \begin{proof}
        Let the partite sets of $\Gamma$ be $V_1, V_2, \cdots, V_m$, where for each integer $i \in [m]$, $V_i = \{u_{i,1}, u_{i,2}, \cdots, u_{i,n_i}\}$. Suppose to the contrary that $\Gamma$ is unbalanced, and then we consider the following two cases.

        {\flushleft {\it Case 1.} $k=1$.}

        In this case, $\Gamma$ is $C_{3}^{-}$-free. By Lemma \ref{Z82}, up to switching equivalence, we may assume that $\sigma (u_{1,1}u_{i,j})=+1$ for integer $i \in [2,m]$ and $j \in [n_{i}]$, $\sigma (u_{2,1}u_{1,z})=+1$ for integer $z \in [n_{1}]$. Since $\Gamma$ is $C_{3}^{-}$-free, all edges in $E(V_i,V_j)$ are positive for integers $i, j \in [2,m]$, and further, all edges in $E(V_1,V_i)$ are positive for integers $i\in [2,m]$. Thus, $\Gamma$ is balanced. 

        {\flushleft {\it Case 2.} $k \ge 2$.}

        In this case, if $\Gamma$ does not contains a negative triangle as a signed subgraph, then similarly to Case 1, $\Gamma$ is a balanced signed graph. Therefore, we may assume that $\Gamma$ contains a negative triangle as a signed subgraph. Without loss of generality, we assume the vertex set of the negative triangle is $\{ u_{1,1}, u_{2,1}, u_{3,1} \}$. By Lemma \ref{Z82}, up to switching equivalence, we may assume that $\sigma (u_{1,1}u_{i,j})=+1$ for integer $i \in \{2,3\}$ and $j \in [n_{i}]$, $\sigma (u_{2,1}u_{1,z})=+1$ for integer $z \in [n_1]$, and $\sigma (u_{2,1}u_{3,1})=-1$. 

        \begin{claim}\label{claim::3.0}
            Any path of length $2k-3$ in $\Gamma[V_1\setminus\{u_{1,1}\},V_2\setminus\{u_{2,1}\}]$ is a negative path.
        \end{claim}
        \begin{proof}
           Suppose to the contrary that $u_{2,2}\dot{P}_{2k-2}'u_{1,2}$ is a positive path of length $2k-3$ in $\Gamma[V_1\setminus\{u_{1,1}\},V_2\setminus\{u_{2,1}\}]$. Then $u_{3,1}u_{1,1}(u_{2,2}\dot{P}_{2k-2}'u_{1,2})u_{2,1}u_{3,1}$ forms a negative cycle of length $2k+1$ in $\Gamma$, which contradicts the assumption.
        \end{proof}

        \begin{claim}\label{claim::3.1}
            All edges in $E(u_{3,1},V_1)$ are positive edges. 
        \end{claim}

        \begin{proof}
            Suppose to the contrary that we assume $\sigma(u_{3,1}u_{1,2})=-1$. Choose an arbitrary path of length $2k-3$ with endpoints $u_{1,2}$ and $u_{2,2}$ in $\Gamma [V_1 \setminus \{ u_{1,1} \}, V_2 \setminus \{ u_{2,1} \}]$, and denote it by $\dot{P}_{2k-2}^{1}$. By Claim \ref{claim::3.0}, $\dot{P}_{2k-2}^{1}$ is a negative path. Furthermore, we observe that $u_{3,1}u_{2,1}u_{1,1}(u_{2,2}\dot{P}_{2k-2}^{1}u_{1,2})u_{3,1}$ forms a negative cycle of length $2k+1$ in $\Gamma$, which is a contradiction. Therefore, all edges in $E(u_{3,1},V_1)$ are positive edges.
        \end{proof}

        \begin{claim}\label{claim::3.2}
            All edges in $E(V_1 \setminus \{ u_{1,1} \} ,V_2 \setminus \{ u_{2,1} \})$ are negative edges. 
        \end{claim}

        \begin{proof}

            By Claim \ref{claim::3.0}, the conclusion is true when $k=2$. We next assume that $k \ge 3$.
            Suppose to the contrary that there exists a positive edge $uv$ in $E(V_1 \setminus \{ u_{1,1} \} ,V_2 \setminus \{ u_{2,1} \})$, where $u\in V_1\setminus\{u_{1,1}\}$ and $v\in V_2\setminus\{u_{2,1}\}$. Take an arbitrary vertex $w \in V_2 \setminus \{u_{2,1}\}$. Then we can choose a path \(\dot{P}_{2k-2}^{2}\) in $\Gamma [V_1 \setminus \{ u_{1,1} \}, V_2 \setminus \{ u_{2,1} \}]$ with $u$ and $w$ as its endpoints, \(uv\in E(\dot{P}_{2k-2}^{2})\), and length \(2k-3\). We first prove that for any vertex $y \in V_1 \setminus (V(\dot{P}_{2k-2}^{2}) \cup \{ u_{1,1} \})$, $\sigma (wy)=+1$. Otherwise, there exists a vertex $y\in V_1\setminus\bigl(V(\dot{P}_{2k-2}^{2})\cup\{u_{1,1}\}\bigr)$ such that $\sigma(wy)=-1$, and then $(v\dot{P}_{2k-2}^{2}w)y$ is a positive path of length $2k-3$ in $\Gamma$, which contradicts Claim \ref{claim::3.0}.

            We next prove that for any vertex $y\in(V_1\setminus\{u_{1,1}\})\cap V(\dot{P}_{2k-2}^{2})$, $\sigma(wy)=+1$. Let $x\in V_1\setminus\bigl(V(\dot{P}_{2k-2}^{2})\cup\{u_{1,1}\}\bigr)$. Since $|V_1 \setminus (V(\dot{P}_{2k-2}^{2}) \cup \{ u_{1,1} \})| \ge k$ and $|V_2 \setminus ((V(\dot{P}_{2k-2}^{2})\setminus \{w\}) \cup \{ u_{2,1} \})| \ge k$, we can choose a negative path of length $2k-3$ with $w$ as an endpoint that contains the edge $wx$ in $\Gamma\bigl[V_1\setminus\bigl(V(\dot{P}_{2k-2}^{2})\cup\{u_{1,1}\}\bigr),V_2\setminus\bigl((V(\dot{P}_{2k-2}^{2})\setminus\{w\})\cup\{u_{2,1}\}\bigr)\bigr]$, denoted by $\dot{P}_{2k-2}^{3}$. Let $u'$ be the other endpoint of $\dot{P}_{2k-2}^{3}$, and further let $v'$ be a neighbor of $u'$ in $V_2\setminus(V(\dot{P}_{2k-2}^{2})\cup\{u_{2,1}\})$. Note that $\sigma (wx)=+1$. Thus, $u'v'$ is a positive edge; otherwise, $(y\dot{P}_{2k-2}^{3}u')v'$ is a positive path of length $2k-3$, which contradicts Claim \ref{claim::3.0}. Now we can choose a negative path $\dot{P}_{2k-2}^{4}$ in $\Gamma\bigl[V_1\setminus\bigl(V(\dot{P}_{2k-2}^{2})\cup\{u_{1,1}\}\bigr),V_2\setminus\bigl((V(\dot{P}_{2k-2}^{2})\setminus\{w\})\cup\{u_{2,1}\}\bigr)\bigr]$ with endpoints $u'$ and $w$, $u'v' \in E(\dot{P}_{2k-2}^{4})$, and length $2k-3$. For any vertex $y\in (V_1\setminus\{u_{1,1}\})\cap V(\dot{P}_{2k-2}^{2})$, we  have $\sigma(wy)=+1$; otherwise, the path $v'\dot{P}_{2k-2}^{4}wy$ is a positive path of length $2k-3$ in $\Gamma$, which contradicts Claim \ref{claim::3.0}. Thus, all edges in $E(w, V_1 \setminus \{ u_{1,1} \})$ are positive edges. Since $w$ is chosen arbitrarily, all edges in $E(V_1 \setminus \{ u_{1,1} \} ,V_2 \setminus \{ u_{2,1} \})$ are positive edges, which contradicts Claim \ref{claim::3.0}.
        \end{proof}

        \begin{claim}\label{claim::3.3}
            All edges in $E(V_2 \setminus \{ u_{2,1} \} ,V_3 \setminus \{ u_{3,1} \})$ are positive edges. 
        \end{claim}

        \begin{proof}
            Suppose to the contrary that we assume $\sigma(u_{2,2}u_{3,2})=-1$. From Claim \ref{claim::3.2}, we can choose a positive path \(\dot{P}_{2k-1}\) of length \(2k-2\) in $\Gamma [V_1 \setminus \{ u_{1,1} \}, V_2 \setminus \{ u_{2,1} \}]$ with $u_{2,2}$ and $u_{2,3}$ as its endpoints. Then $u_{1,1}(u_{2,3}\dot{P}_{2k-1}u_{2,2})u_{3,2}u_{1,1}$ forms a negative cycle of length $2k+1$ in $\Gamma$, which is a contradiction. 
        \end{proof}

        By Claim \ref{claim::3.2}, there is a negative path \(\dot{P}_{2k}^{1}\) of length \(2k-1\) in $\Gamma [V_1, V_2]$ with $u_{1,1}$ and $u_{2,1}$ as its endpoints. Thus $\sigma (u_{2,1}u_{3,2})=-1$, as otherwise $u_{3,2}(u_{1,1}\dot{P}_{2k}^{1}u_{2,1})u_{3,2}$ forms a negative cycle of length $2k+1$ in $\Gamma$. On the other hand, by Claim \ref{claim::3.3}, there is a positive path \(\dot{P}_{2k}^{2}\) of length \(2k-1\) in $\Gamma [V_2 \setminus \{ u_{2,1} \} \cup V_3 \cup \{ u_{1,1} \}]$ with $u_{1,1}$ and $u_{3,2}$ as its endpoints. It is easy to check that $u_{2,1}(u_{3,2}\dot{P}_{2k}^{2}u_{1,1})u_{2,1}$ forms a negative cycle of length $2k+1$ in $\Gamma$, which is a contradiction. This completes the proof.
    \end{proof}
	
    \section{Proof of the Theorem \ref{thm::main}}

    Let $\dot{G}=(G,\sigma)$ be a signed graph satisfying the conditions of Theorem \ref{thm::main} with the maximum number of edges. It is easy to see that the underlying graph $G$ is connected and satisfies $e(G)\ge e(C_{3}^{-}\cdot T_r(n-2))$. By applying the Erd\H{o}s–Simonovits stability theorem, we can obtain the following lemmas, which describe the structural properties of the underlying graph $G$.

    \begin{lem}\label{G_1}
        For every $\varepsilon >0$, there exists an integer $n_0$ such that if $n \ge n_0$, then we have $e(G) \ge e (T_{r}(n)) - \varepsilon^2n^2$. Moreover, $G$ has a partition $V(G) = V_1 \cup \cdots \cup V_r$ such that 
        \begin{enumerate}[(i)]
        \item $\sum\limits_{1 \le i < j \le r} e(V_i, V_j)$ attains the maximum; 
        \item $\sum\limits_{i=1}^r e(V_i) \le \varepsilon^2n^2$; 
        \item for each integer $i \in [r]$, $||V_i| - \frac{n}{r}| \le 2\varepsilon n$.
        \end{enumerate}
    \end{lem}

    \begin{proof}
        Let $n_0=\frac{3r-2}{r\varepsilon^2}$. It is straightforward to check that
        \begin{equation}\label{eq::1}
        \begin{aligned}
            e(G)&=e(\dot{G}) \\
            &\ge e(C_{3}^{-} \cdot T_r(n-2)) \\
            &\ge \frac{r-1}{2r}(n-2)^2-\frac{r}{8}+3 \\
            &\ge e(T_{r}(n))-\varepsilon^2n^2.
        \end{aligned}
        \end{equation}
        From Lemma \ref{ESthm}, there exists a vertex partition $V(G) = U_1 \cup \cdots \cup U_r$ with $\left\lfloor \frac{n}{r} \right\rfloor \le |U_i| \le \left\lceil \frac{n}{r} \right\rceil$ such that $\sum_{i=1}^r e(U_i) \le \varepsilon^2n^2$. Furthermore, we choose a partition $V(G) = V_1 \cup \cdots \cup V_r$ such that $\sum_{1 \le i < j \le r} e (V_i, V_j)$ attains the maximum. Hence,  $\sum_{i=1}^r e(V_i) \le \sum_{i=1}^r e(U_i) \le \varepsilon^2n^2$. Next, we prove that for each $i \in [r]$, $\left||V_i| - \frac{n}{r}\right| \le 2\varepsilon n$. Let $a=\max \left\{ \left||V_j| - \frac{n}{r}\right|, i \in [r] \right\}$. We have

        \begin{equation}\label{eq::2}
        \begin{aligned}
            e(G) &= \sum_{1 \leqslant i < j \le r} e (V_i, V_j) + \sum_{i=1}^r e(V_i) \\
            & \le \sum_{1 \leqslant i < j \le r} |V_i||V_j| + \varepsilon^2n^2\\
            & =|V_1|(n - |V_1|) + \sum_{2 \le i < j \le r} |V_i||V_j| + \varepsilon^2n^2\\
            &= |V_1|(n - |V_1|) + \frac{1}{2} \left( ( \sum_{i=2}^r |V_i| )^2 - \sum_{i=2}^r |V_i|^2\right) + \varepsilon^2 n^2\\
            &\le |V_1|(n - |V_1|) + \frac{1}{2}(n - |V_1|)^2 - \frac{1}{2(r - 1)}(n - |V_1|)^2 + \varepsilon ^2 n^2 \\
            &< \frac{r-1}{2r}n^2 - \frac{r}{2r-2}a^2 + \varepsilon^2 n^2,
        \end{aligned}
        \end{equation}
        where the second last inequality holds by H\H{o}lder’s inequality. 

        Combining \eqref{eq::1} and \eqref{eq::2}, we conclude that $\max \left\{ \left||V_j| - \frac{n}{r}\right|, i \in [r] \right\} \le 2\varepsilon n$.
    \end{proof}

    Let $\varepsilon$ be a sufficiently small positive constant. We denote $L:=\left\{ v \in V(G) \mid d(v) \leq \left(1 - \frac{1}{r} - 4\varepsilon\right) n \right\}$, and $W := \bigcup_{i=1}^r\left\{ v \in V_i \mid d_{V_i}(v) \ge 2\varepsilon n \right\}$. 

    \begin{lem}\label{G_2}
        The following statements hold.
        \begin{enumerate}[(i)]
            \item $|L| \le \varepsilon n$;
            \item $|W| \le \varepsilon n$.
        \end{enumerate}
    \end{lem}

    \begin{proof}
        We first show that $|L| \le \varepsilon n$. Suppose to the contrary that $|L| > \varepsilon n$. Then there exists a subset $N \subseteq L$ with $|N| = \lfloor\varepsilon n \rfloor$. Thus, we have 
        \begin{equation}
        \begin{aligned}
            e(G - N) &\ge e(G) - \sum_{v \in N} d_G(v) \\
            &\ge e\left(T_{r}(n)\right) - \varepsilon^2n^2 - \varepsilon n \left(1 - \frac{1}{r} - 4\varepsilon \right)n\\
            &\ge \frac{1}{2}\left(1-\frac{1}{r}\right)n^2-\frac{r}{8}-\varepsilon^2n^2-\varepsilon n \left(1 - \frac{1}{r} - 4\varepsilon \right)n\\
            &= \frac{1}{2}\left(1-\frac{1}{r}\right)n^2-\left(1-\frac{1}{r}\right)\varepsilon n^2+3\varepsilon^2n^2-\frac{r}{8}\\
            &> \frac{1}{2}\left(1-\frac{1}{r}\right)(n-\varepsilon n+2)^2\\
            &> \frac{1}{2}\left(1-\frac{1}{r}\right)(n-\lfloor\varepsilon n\rfloor+1)^2\\
            &\ge e(T_{r}(n-\lfloor\varepsilon n\rfloor)).
        \end{aligned}
        \nonumber
        \end{equation}
        By Lemma \ref{Sthm}, $G - N$ contains a copy of $F$, a contradiction. Thus, $|L| \le \varepsilon n$.

        Next we show that $|W| \le \varepsilon n$. By Lemma \ref{G_1}, we have $\sum_{i=1}^r e(V_i) \le \varepsilon^2 n^2$. Let $W_i = W \cap V_i$ for all $i \in [r]$. Then
        \begin{equation}
        \begin{aligned}
            \sum_{i=1}^r e(V_i) &= \sum_{i=1}^r \left( \frac{1}{2}\sum_{v \in V_i} d_{V_i}(v) \right) \\
            &\ge \frac{1}{2}\sum_{i=1}^r \sum_{v \in W_i} d_{V_i}(v) \\
            &\ge |W| \varepsilon n .
            \end{aligned}
            \nonumber
        \end{equation}
        Thus, $|W| \le \varepsilon n$.
    \end{proof}

    For convenience, let $|V(F)| = \ell$.

    \begin{lem}\label{G_3}
        $W \subseteq L$.
    \end{lem}
    
    \begin{proof}
        Suppose to the contrary that there exists a vertex \( u_{1,1} \in W \) and \( u_{1,1} \notin L \). Without loss of generality, assume \( u_{1,1} \in V_1 \). Since \( V(G) = V_1 \cup \cdots \cup V_r \) is the partition that maximizes \( \sum_{1 \leq i < j \leq r} e(V_i, V_j) \), it follows that \( d_{V_1}(u_{1,1}) \leq d_{V_i}(u_{1,1}) \) for every \( i \in [2, r] \). This implies \( d(u_{1,1}) \geq r d_{V_1}(u_{1,1}) \), or equivalently \( d_{V_1}(u_{1,1}) \leq d(u_{1,1})/r \). Furthermore, since \( u_{1,1} \notin L \), we have 
        \begin{equation}\label{eq::3}
        \begin{aligned}
        d_{V_2}(u_{1,1}) &\geq d(u_{1,1}) - d_{V_1}(u_{1,1}) - (r - 2)\left( \frac{1}{r} + 2\varepsilon \right)n \\
        &\geq \left( 1 - \frac{1}{r} \right)d(u_{1,1}) - (r - 2)\left( \frac{1}{r} + 2\varepsilon \right)n \\
        &> \left( 1 - \frac{1}{r} \right)\left(1 - \frac{1}{r} - 4\varepsilon\right) n - (r - 2)\left( \frac{1}{r} + 2\varepsilon \right)n \\
        &= \frac{n}{r^2} - (2r - \frac{4}{r})\varepsilon n.
        \end{aligned}
        \end{equation}
        Note that $|L| \le \varepsilon n$ and $|W| \le \varepsilon n$. For sufficiently large n, we have
        $$
        |V_i \setminus (L \cup W)| \geq \left( \frac{1}{r} - 2\varepsilon \right)n - 2\varepsilon n\geq \ell.
        $$
        We claim that $u_{1,1}$ has no neighbors in $V_1 \setminus (W \cup L)$. Assume $u_{1,2}$ be a neighbor of $u_{1,1}$ in $ V_1 \setminus (W \cup L) $. Let $ u_{1,3}, \cdots, u_{1,\ell} $ be a subset of $ V_1 \setminus (W \cup L) $ that does not include $u_{1,1}$ and $u_{1,2}$. According to the definitions of $L$ and $W$, for any $i \in [2,\ell]$, we have $(1 - \frac{1}{r} - 4\varepsilon) n$, and $d_{V_1}(u_{1,i})<2 \varepsilon n$. Thus, 
        \begin{equation}\label{eq::4}
        \begin{aligned}
        d_{V_2}(u_{1,i}) &\geq d(u_{1,i}) - d_{V_1}(u_{1,i}) - (r - 2)\left( \frac{1}{r} + 2\varepsilon \right)n \\
        &\geq (1 - \frac{1}{r} - 4\varepsilon) n - 2 \varepsilon n - (r - 2)\left( \frac{1}{r} + 2\varepsilon \right)n \\
        &= \frac{n}{r}-(2r+2)\varepsilon n.
        \end{aligned}
        \end{equation}
        
        Combining \eqref{eq::3} and \eqref{eq::4}, by Lemma \ref{lem::jh}, we have
        
        \begin{equation}
        \begin{aligned}
        \left|  (\bigcap_{\substack{i \in [\ell] }} N_{V_{2}}(u_{1,i}))  \setminus (L \cup W)) \right| 
        &\geq d_{V_2}(u_{1,1}) + \sum_{i=2}^{\ell} d_{V_2}(u_{1,i}) - (\ell - 1)|V_2| - |L| - |W| \\
        &>\frac{n}{r^2} - (2r - \frac{4}{r})\varepsilon n + (\ell - 1)\left(\frac{n}{r}-(2r+2)\varepsilon n\right)\\
        &~~~~-(\ell - 1)\left(\frac{n}{r}+2\varepsilon n\right)-2\varepsilon n\\
        &= \frac{n}{r^2} - \left( 2r\ell+4\ell - \frac{4}{r} - 2\right)\varepsilon n.
        \end{aligned}
        \nonumber
        \end{equation}
        This implies that, for sufficiently large $n$, there exist $\ell$ vertices $u_{2,1}, u_{2,2}, \cdots, u_{2,\ell} $ in $V_2 \setminus (L \cup W)$ such that the subgraph induced by $\{u_{1,1}, u_{1,2}, \cdots, u_{1,\ell}\}$ and $\{u_{2,1}, u_{2,2}, \cdots, u_{2,\ell}\}$ contains a complete bipartite graph. Assume, for any integer $t$ with $2 \leq t \leq r - 1$, that the subgraph induced by $\{u_{1,1}, u_{1,2}, \cdots, u_{1,\ell}\}, \{u_{2,1}, \cdots, u_{2,\ell}\}, \cdots, \{u_{t,1}, \cdots, u_{t,\ell}\}$ contain a complete $t$-partite graph. Analogously to \eqref{eq::3} and \eqref{eq::4}, we obtain that, 
        $$
        d_{V_{t+1}}(u_{1,1}) > \frac{n}{r^2} - (2r - \frac{4}{r})\varepsilon n,
        $$
        and for each $i \in [t]$ and $j \in [\ell]$(with $u_{i,j} \neq u_{1,1}$),
        $$
        d_{V_{t+1}}(u_{i,j}) > \frac{n}{r} - (2r+2)\varepsilon n.
        $$
        Consider the common neighbors of these vertices in $V_{t+1}$. By Lemma \ref{lem::jh}, we have
        \begin{equation}
        \begin{aligned}
        \left|  (\bigcap_{\substack{i \in [t] , \\ j \in [\ell]}} N_{V_{t+1}}(u_{i,j}))  \setminus (L \cup W)) \right| &\ge \frac{n}{r^2} - \left(2r - \frac{4}{r}\right)\varepsilon n + (t\ell-1)\left(\frac{n}{r}-(2r+2)\varepsilon n\right)\\
        &~~~~-(t\ell-1)\left(\frac{n}{r}+2\varepsilon n\right) - 2\varepsilon n\\
        &= \frac{n}{r^2}-\left(  2rt\ell + 4t\ell-2-\frac{4}{r}\right)\varepsilon n \\
        &> \ell,
        \end{aligned}
        \nonumber
        \end{equation}
        for sufficiently large $n$. Thus, there exist $\ell$ vertices $u_{t+1,1},u_{t+1,2}, \cdots, u_{t+1,\ell}$ such that the subgraph induced by $ \{u_{1,1}, u_{1,2}, \cdots, u_{1,\ell}\} \cup \{u_{2,1}, u_{2,2}, \cdots, u_{2,\ell}\} \cup \cdots \cup \{u_{t+1,1}, u_{t+1,2}, \cdots, u_{t+1,\ell}\}$ contains a copy of complete $t+1$-partite graph. Furthermore, for each $i \in [2,\ell]$, there exist $\{u_{i,1}, u_{i,2}, \cdots, u_{i,\ell}\} \subseteq V_i \setminus (L \cup W)$ such that the subgraph induced by $ \{u_{1,1}, u_{1,2}, \cdots, u_{1,\ell}\} \cup \{u_{2,1}, u_{2,2}, \cdots, u_{2,\ell}\} \cup \cdots \cup \{u_{r,1}, u_{r,2}, \cdots, u_{r,\ell}\}$ contains a copy of complete $r$-partite graph. Note that $u_{1,1}$ is adjacent to $u_{1,2}$ by assumption. It follows that $G$ contains a copy of $F$, a contradiction. Therefore, $u_{1,1}$ has no neighbors in $V_1 \setminus (W \cup L)$. 
        
        Consider the neighbors of $u_{1,1}$ in $V_{1}$. We have
        $$
        d_{V_1}(u_{1,1}) \le |L| + |W| - 1 < 2 \varepsilon n,
        $$
        which contradicts our assumption that $u_{1,1} \in W$. 
    \end{proof}

    We denote $S_{i}:=V_{i} \setminus L$. 

    \begin{lem}\label{G_4}
        For each $i \in [r]$, $e(G[S_i])=0$.
    \end{lem}

    \begin{proof}
        Suppose to the contrary that there exists some $i$ such that $e(G[S_i]) \ge 1$. Without loss of generality, let $e(G[S_1]) \ge 1$. By Lemma \ref{G_1}(\romannumeral 3) and Lemma \ref{G_2}(\romannumeral 1), for sufficiently large $n$, we have 
        $$
        |S_1| \geq \left( \frac{1}{r} - 2\varepsilon \right)n - \varepsilon n\geq \ell.
        $$
        Thus, there exist $\ell$ vertices $u_{1,1}, u_{1,2}, \cdots, u_{1,\ell} $ in $S_1$ such that the subgraph induced by $\{u_{1,1}, u_{1,2}, \cdots, u_{1,\ell} \}$ contains at least one edge. By using the similar method as in Lemma \ref{G_3}, we have 
        $$
        d_{V_2}(u_{1,i}) \ge \frac{n}{r}-(2r+2)\varepsilon n.
        $$
        Using Lemma \ref{lem::jh}, we get
        \begin{equation}
        \begin{aligned}
        \left|  \bigcap_{\substack{i \in [\ell] }} N_{S_{2}}(u_{1,i}))  \right|
        & = \left|  (\bigcap_{\substack{i \in [\ell] }} N_{V_{2}}(u_{1,i}))  \setminus L ) \right|\\
        &\ge \sum_{i=1}^{\ell} d_{V_2}(u_{1,i}) - (\ell - 1)|V_2| - |L| \\
        &> \ell \left(\frac{1}{r}-(2r+2)\varepsilon \right)n-(\ell - 1)\left(\frac{1}{r}+2\varepsilon \right)n-\varepsilon n\\
        &= \frac{n}{r} - \left( 2r\ell+4\ell - 1\right)\varepsilon n.
        \end{aligned}
        \nonumber
        \end{equation}
        For sufficiently large $n$, it follows that there exist $\ell$ vertices $u_{2,1}, u_{2,2}, \cdots, u_{2,\ell} $ in $S_2$ such that the subgraph induced by $\{u_{1,1}, u_{1,2}, \cdots, u_{1,\ell}\}$ and $\{u_{2,1}, u_{2,2}, \cdots, u_{2,\ell}\}$ contains a complete bipartite graph. By an analogous argument to that in Lemma \ref{G_3}, for each $i \in [2,r]$, there exist $\{u_{i,1}, u_{i,2}, \cdots, u_{i,\ell}\} \subseteq S_i$ such that the subgraph induced by $ \{u_{1,1}, u_{1,2}, \cdots, u_{1,\ell}\} \cup \{u_{2,1}, u_{2,2}, \cdots, u_{2,\ell}\} \cup \cdots \cup \{u_{r,1}, u_{r,2}, \cdots, u_{r,\ell}\}$ contains a copy of complete $r$-partite graph. Recall that the subgraph induced by $\{u_{1,1}, u_{1,2}, \cdots, u_{1,\ell} \}$ contains at least one edge. Thus, $G$ contains a copy of $F$, which is a contradiction. 
    \end{proof}

    The following conclusion can be obtained. As the proof follows the same line of reasoning as those of Lemma \ref{G_3} and Lemma \ref{G_4}, we omit it here.

    \begin{lem}\label{G_5}
        %(i) For each $i \in [r]$, any subset of $S_i$ of size $O(1)$ can be a part of some complete multipartite graph. Moreover, in this complete multipartite graph contains exactly one part of size \(\theta(n)\), with all other parts of size \(O(1)\). 

        (i) For any vertices $u,v\in S$, they have $\Theta(n)$ common neighbors in $S$. (ii) Any edge $e$ in $E(G[S])$ is contained in a complete multipartite graph. Moreover, this complete multipartite graph contains exactly one part of size \(\Theta(n)\), with all other parts of size \(O(1)\). 
    \end{lem}

    Next, we show that $\dot{G} [S]$ is balanced. To do this, we first prove a weaker result. 

    \begin{lem}\label{S_2}
        $\dot{G} [S]$ is $C_{3}^{-}$-free. 
    \end{lem}
    \begin{proof}
        Suppose to the contrary that $\dot{G} [S]$ contains a negative triangle $C_{3}^{-}$. Let $V(C_{3}^{-}) = \{ u,v,w \}$. Applying Lemma \ref{G_5}(ii), $uv$ is contained in a complete $3$-partite graph $K$ with each part of size at least $2k$. Note that $\dot{G} [S]$ is $C_{2k+1}^{-}$-free. Thus, we have $w \notin V(K)$; otherwise this contradicts Lemma \ref{S_1}. By Lemma \ref{Z82} and \ref{S_1}, up to switching equivalence, we may assume $(K,\sigma)$ is all positive. Since $uwv$ forms a negative path of length $2$ in $\dot{G} [S]$, $(K \cup \{w \},\sigma)$ contains a negative cycle of length $2k+1$ in $\dot{G} [S]$, which is a contradiction.
    \end{proof}

    \begin{lem}\label{S_3}
        $\dot{G} [S]$ is balanced. 
    \end{lem}
    
    \begin{proof}
        Since $|S_{1}| = |V_1 \setminus L| \ge \left( \frac{1}{r} - 3\varepsilon \right)n$, for sufficiently large $n$, we may choose $2k$ vertices $u_{1,1},u_{1,2}, \cdots, u_{1,2k}$ from $S_{1}$. Let $U_1=u_{1,1},u_{1,2}, \cdots, u_{1,2k}$. By using the similar method as in Lemma \ref{G_3}, for each $i \in [2k]$, we have 
        $$
        d_{V_2}(u_{1,i}) \ge \frac{n}{r}-(2r+2)\varepsilon n.
        $$
        Consider the common neighbors of vertices of $U_1$ in $S_2$. By Lemma \ref{lem::jh}, we have 
        \begin{equation}
        \begin{aligned}
        \left|  (\bigcap_{\substack{i \in [2k] }} N_{S_{2}}(u_{1,i})) \right|
        &=\left|  (\bigcap_{\substack{i \in [2k] }} N_{V_{2}}(u_{1,i}))  \setminus L ) \right|\\ 
        &\geq \sum_{i=1}^{2k} d_{V_2}(u_{1,i}) - (2k - 1)|V_2| - |L| \\
        &>2k\left(\frac{n}{r}-(2r+2)\varepsilon n\right)-(2k - 1)\left(\frac{n}{r}+2\varepsilon n\right)-\varepsilon n\\
        &= \frac{n}{r} - \left( 4kr+8k - 1\right)\varepsilon n\\
        &> 2k,
        \end{aligned}
        \nonumber
        \end{equation}
        for sufficiently large $n$. It follows that there exist $2k$ vertices $u_{2,1}, u_{2,2}, \cdots, u_{2,2k} $ in $S_2$ that are the common neighbors of $u_{1,1}, u_{1,2}, \cdots, u_{1,2k} $. By an analogous argument to that in Lemma \ref{G_3}, for each $i \in [2,r-1]$, there exist $u_{i,1}, u_{i,2}, \cdots, u_{i,\ell}$ in $ S_i $ that are the common neighbors of $u_{1,1}, \cdots, u_{1,2k}, \cdots, u_{i-1,1}, \cdots, u_{i-1,2k}$. Let $U_i = \{ u_{i,1},u_{i,2},\cdots,u_{i,2k} \}$. Next we consider the common neighbors of vertices of $\cup_{i=1}^{r-1}U_{i}$ in $S_r$. Then we get 
        \begin{equation}
        \begin{aligned}
        \left|  \bigcap_{\substack{i \in [r-1] , \\ j \in [2k]}} N_{S_{r}}(u_{i,j})\right|&=\left|  (\bigcap_{\substack{i \in [r-1] , \\ j \in [2k]}} N_{V_{r}}(u_{i,j}))  \setminus L ) \right| \\
        &\ge  2k(r-1)\left(\frac{n}{r}-(2r+2)\varepsilon n\right)\\
        &~~~~-(2k(r-1)-1)\left(\frac{n}{r}+2\varepsilon n\right) - \varepsilon n\\
        &= \frac{n}{r}-\left( 4k(r-1)(r+2)-1) \right)\varepsilon n.
        \end{aligned}
        \nonumber
        \end{equation}
        
        Let $U_r = \bigcap_{\substack{i \in [r-1], j \in [2k]}} N_{S_{r}}(u_{i,j})$. Note that $\dot{G}$ is $C_{2k+1}^{-}$-free, by Lemma \ref{S_1} and Lemma \ref{G_4}, the signed subgraph induced by $\cup_{i=1}^{r}U_{i}$ is a balanced complete $r$-partite graph. Up to switching equivalence, we may assume $\dot{G}[\cup_{i=1}^{r}U_{i}]$ is all positive. For each $i \in [r]$, let $\overline{U}_{i} = S_{i} \setminus U_{i}$.
    \begin{claim}\label{S_3_1}
        For any integer $i$ with $1\le i \le r-1$ and any $v \in \overline{U}_{i}$, all edges in $E(v,U_r)$ have the same sign. 
    \end{claim}
    \begin{proof}
        For each $v \in \overline{U}_{i}$, we have
        \begin{equation}
        \begin{aligned}
        |N_{U_r}(v)|=d_{U_r}(v) &\geq d_{S_r}(v) - (|S_r| - |U_r|) \\
        &\geq (\frac{1}{r} - (2r+3)\varepsilon) n - (\frac{1}{r}+2\varepsilon)n+ \frac{n}{r}-\left( 4k(r-1)(r+2)-1) \right)\varepsilon n\\
        &= \frac{n}{r} - (4kr-4k+2)(r+2)\varepsilon n\\
        & \ge 2.
        \end{aligned}
        \nonumber
        \end{equation}
        If there exist two edges with opposite signs in $E(v,U_r)$, then we can find a negative cycle of length $2k+1$ in $\dot{G}[\cup_{i=1}^{r}U_{i} \cup \{ v \}]$. This is a contradiction.
    \end{proof}
    
    Let $\dot{G}'$ be a signed subgraph of $\dot{G}$ obtained by adding the vertex set $\cup_{i=i}^{r-1}\overline{U}_{i}$ and the edge set $E_{\dot{G}}(\cup_{i=i}^{r-1}\overline{U}_{i},U_r)$ to $\dot{G}[\cup_{i=1}^{r}U_{i}]$. Note that $\dot{G}[\cup_{i=1}^{r}U_{i}]$ is all positive. By Claim \ref{S_3_1}, $\dot{G}'$ contains no negative cycles. Thus, $\dot{G}'$ is balanced. Up to switching equivalence, we may assume that $\dot{G}'$ is all positive.   

    \begin{claim}\label{S_3_3}
        $\dot{G} [S \setminus \overline{U}_{r}]$ is balanced. 
    \end{claim}

    \begin{proof}
        If $\dot{G} [S \setminus \overline{U}_{r}]$ is unbalanced, then there exists a negative edge $uv$. Note that $\dot{G}'$ is all positive, which implies that $uv \in \cup_{1\le i <j \le r-1} E(\overline{U}_{i},\overline{U}_{j})$. Consider the common neighbors of vertices of $u$ and $v$ in $S_{r}$, by the similar method as Lemma \ref{G_3}, we have
        \begin{equation}
        \begin{aligned}
        \left| N_{S_r}(u) \cap N_{S_r}(v) \right| 
        &\geq d_{V_r}(u) + d_{V_r}(v) - |V_2| - |L| \\
        &> 2 \left(\frac{1}{r}-(2r+2)\varepsilon \right)n-\left(\frac{1}{r}+2\varepsilon \right)n-\varepsilon n\\
        &= \frac{n}{r} - \left( 4r+7 \right)\varepsilon n.
        \end{aligned}
        \nonumber
        \end{equation}
        Recall that $|U_r| \ge \frac{n}{r}-\left( 4k(r-1)(r+2)-1) \right)\varepsilon n$ and $S_r \le (\frac{1}{r}+2\varepsilon) n$, we get  
        $$
        | N_{S_r}(u) \cap N_{S_r}(v) | + |U_r| >|S_r|.
        $$
        According to the pigeonhole principle, $N_{U_r}(u) \cap N_{U_r}(v)=N_{S_r}(u) \cap N_{S_r}(v) \cap U_r \neq \emptyset$. Let $w \in N_{U_r}(u) \cap N_{U_r}(v)$. Since $\dot{G}'$ is all positive, we have $\sigma(vw) = \sigma(uw)=+1$. Thus, $uvw$ forms a negative triangle in $\dot{G}[S]$, which contradicts Lemma \ref{S_2}. 
    \end{proof}

    Now, we show that $\dot{G}[S]$ is balanced. By Claim \ref{S_3_3}, up to switching equivalence, we may assume that $\dot{G} [S \setminus \overline{U}_{r}]$ is all positive. For any vertex $v \in \overline{U}_{r}$, if there exist two edges with opposite signs in $E(v, S \setminus \overline{U}_{r})$, then by Lemma \ref{G_5}(i), we can find a negative cycle of length $2k+1$ in the subgraph $\dot{G}\left[(S \setminus \overline{U}_{r}) \cup \{ v \}\right]$. It follows that for any vertex $v \in \overline{U}_{r}$, all edges in $E(v, S \setminus \overline{U}_{r})$ have the same sign. Therefore, $\dot{G}[S]$ is balanced. The proof is completed. 
    \end{proof}

    \begin{lem}\label{S_4}
        We denote $L':=\left\{ v \in V(G) \mid d(v) \leq \left(1 - \frac{1}{r} - 7\varepsilon\right) n \right\}$, $S_{i}' := V_{i} \setminus L'$, and $S' := \cup_{i=1}^{r} S_{i}'$. The following statements hold:
         \begin{enumerate}[(i)]
            \item for any negative cycle $C^{-}$ in $\dot{G}$, $L' \setminus V(C^{-}) = \emptyset$;
            \item $e(S_{i}')=0$;
            \item $\dot{G}[S']$ is balanced.
        \end{enumerate}
    \end{lem}

    \begin{proof}

    Let  $C^-$  be an arbitrary negative cycle. We make the following claim.
        \begin{claim}\label{S_4_1}
            For each $v \in L \setminus V(C^{-})$, $d(v) > (1-\frac{1}{r}-7\varepsilon)n$.
        \end{claim}
        \begin{proof}
            Suppose to the contrary that there exists a vertex $v \in L \setminus V(C^{-})$ with $d(v) \le (1-\frac{1}{r}-7\varepsilon)n$. Let $u$ be a vertex in $S_1$. By Lemma \ref{S_3}, up to switching equivalence, we may assume that $\dot{G}[S]$ is all positive. Let $G'$ be the graph with $V(G') = V(G)$ and edge set $E(G') = E(G \setminus \{v\}) \cup \{vw\mid w \in \cup_{i=2}^{r}N_{S_i}(u)\}$. Let $\dot{G}' = (G',\sigma')$ be the signed graph satisfying
        \begin{equation}
			\sigma'(e)=\left\{
			\begin{array}{ll}
				\sigma(e),&\mbox{if $e\in E(G \setminus \{v\})$ },\\
				+1,&\mbox{if $e \in \{vw\mid w \in \cup_{i=2}^{r}N_{S_i}(u)\}$.}
			\end{array}
			\right.
			\nonumber
		\end{equation}
        
        We first claim that $G'$ is $F$-free. Assume $G'$ contains a copy of $F$, denoted as $F'$. Hence,  $v \in V(F')$. Let $N_{G'}(v) \cap V(F') = \{w_1, \cdots, w_t\}$. Clearly, for each \(i\in [t]\), \(w_i\notin S_1\). Thus, we have 
        \begin{equation}
            \begin{aligned}
            d_{V_1}(w_{i}) &\geq d(w_{i}) - d_{V \setminus V_1}(w_{i})\\
            &>\left(1 - \frac{1}{r} - 4\varepsilon\right) n - (r - 2)\left( \frac{1}{r} + 2\varepsilon \right)n \\
            &= \frac{n}{r} - 2r\varepsilon n.
            \end{aligned}
            \nonumber
        \end{equation}
        Now, we consider the common neighbors of these vertices in $S_{1}$. By Lemma \ref{lem::jh}, we have
        \begin{equation}
        \begin{aligned}
        \left|  (\bigcap_{\substack{i \in [t] }} N_{S_{1}}(w_{i})) \right|
        &=\left|  (\bigcap_{\substack{i \in [t] }} N_{V_{1}}(w_{i}))  \setminus L ) \right|\\ 
        &\geq \sum_{i=1}^{t} d_{V_1}(w_{i}) - (t - 1)|V_1| - |L| \\
        &>t\left(\frac{n}{r}-2r\varepsilon n\right)-(t - 1)\left(\frac{n}{r}+2\varepsilon n\right)-\varepsilon n\\
        &= \frac{n}{r} - \left( 2tr+2t - 1\right)\varepsilon n.
        \end{aligned}
        \nonumber
        \end{equation}
        For sufficiently large $n$, it follows that there exist a vertex $v' \in S_1$ such that $v'$ is a common neighbor of $w_1, w_2, \cdots, w_t$. Then $(F'\setminus \{v\}) \cup \{ v' \}$ forms a copy of $F$ in $G$, which is a contradiction.

        Next, we claim that $\dot{G}'$ is a $C_{2k+1}^{-}$-free unbalanced signed graph. Obviously, $\dot{G}'$ is unbalanced. Suppose to the contrary that $\dot{G}'$ contains $C_{2k+1}^{-}$ as a signed subgraph. Hence, $v \in V(C_{2k+1}^{-})$. Let $N(v) \cap V(C_{2k+1}^{-}) = \{ w_1, w_2 \}$. By using the similar method as above, we conclude that there exists a common neighbor $v'$ of $w_1$ and $w_2$ in $S_1$ such that $v'\notin V(C_{2k+1}^{-})$. Note that $\sigma (v'w_1)=\sigma (vw_1)=+1$ and $\sigma (v'w_2)=\sigma (vw_2)=+1$. It follows that $(C_{2k+1}^{-} \setminus \{v \}) \cup \{v' \}$ forms a negative cycle of length $2k+1$ in $\dot{G}$, which is a contradiction. 

        Let $L_i = V_{i} \cap L$. Recall that $u \in S_1 \subseteq V_1$, so we have
        \begin{equation}
        \begin{aligned}
        d(u)
        &= d_{V_1}(u) + \sum_{i \in [2,r]}d_{V_i}(u)\\ 
        &= d_{V_1}(u)+\sum_{i \in [2,r]}(d_{S_i}(u)+d_{L_i}(u))\\
        &\le d_{V_1}(u) +\sum_{i \in [2,r]}d_{S_i}(u) + \sum_{i \in [2,r]}|L_i|\\
        &\le \varepsilon n + \sum_{i \in [2,r]}d_{S_i}(u)+\varepsilon n,
        \end{aligned}
        \nonumber
        \end{equation}
        and 
        \begin{equation}
        \sum_{i \in [2,r]}d_{S_i}(u) \ge d(u)-2\varepsilon n\ge (1-\frac{1}{r}-6\varepsilon)n.
        \nonumber
        \end{equation}
        Therefore, 
        $$
        e(\dot{G}')-e(\dot{G}) \ge (1-\frac{1}{r}-6\varepsilon)n - (1-\frac{1}{r}-7\varepsilon)n >0.
        $$
        We get a contradiction. 
        \end{proof}
        Clearly, every vertex $v \in L'$ satisfies $v \in V(C^{-})$; otherwise, by Claim \ref{S_4_1}, $d(v) > \left(1 - \frac{1}{r} - 7\varepsilon\right)n$. Hence, $L' \setminus V(C^{-}) = \emptyset$. Extending the reasoning from Lemma \ref{G_3}, we find $W \subseteq L'$. Applying an analogous argument from Lemma \ref{G_4}, we give $e(S_{i}')=0$. Similar to Lemmas \ref{G_5}--\ref{S_3}, we conclude that $\dot{G}[S']$ is balanced.
    \end{proof}

    \noindent\textbf{The proof of Theorem \ref{thm::main}.}  Let $L'=\{ v_1,v_2,\cdots, v_{p} \}$. By Lemma \ref{S_4}(iii), up to switching equivalence, we may assume that $\dot{G} [S']$ is all positive. We consider two cases to prove Theorem \ref{thm::main}.
    
    {\flushleft {\it Case A.} There exists a negative cycle $C^{-}$ such that $V(C^{-}) \cap S' \neq \emptyset$.}

    We first state the following claim.

    \begin{claim}\label{main_1}
        For any negative cycle $C^{-}=(C,\sigma)$ satisfying $V(C^{-}) \cap S' \neq \emptyset$, $C\setminus S'$ is a path of length $|L'|-1$.
    \end{claim}

    \begin{proof}
        Suppose to the contrary that there exists a negative cycle $(C, \sigma)$ such that $C \setminus S'$ is the disjoint union of at least two paths. Note that $\dot{G}[S']$ is all positive and $(C,\sigma)$ is a negative cycle. Thus, there exist a negative path $\dot{P}$ with both endpoints in $S'$ in $\dot{G}[L'\cup(\cup_{i=1}^{p}N(v_{i})\cap S')]$, such that $L'\setminus V(\dot{P})\neq\emptyset$. According to Lemma \ref{G_5}(i), there exists a negative cycle $C^{-}{'}$ in $\dot{G}[S'\cup V(\dot{P})]$. It follows that $L'\setminus V(C^{-}{'})=L'\setminus V(\dot{P})\neq\emptyset$, which contradicts Lemma \ref{S_4}(i).
    \end{proof}
    
    Choose the shortest one among these negative cycles, and without loss of generality, denote it by $C^{-} = (C, \sigma)$. Applying Claim \ref{main_1}, $C \setminus S'$ is a path of length $|L'|-1$, and we may assumet that its vertices are ordered as $v_1 v_2 \cdots v_p$. We consider the following two cases. 
    
    {\flushleft {\it Subcase A.1.} $|V(C^{-}) \cap S'| \ge 2$.}

    Let $N(v_1) \cap C^{-} \cap S' = \{ v_{0} \}$ and $N(v_p) \cap C^{-} \cap S' = \{ v_{p+1} \}$. Up to switching equivalence, we may assume that $\sigma(v_i v_{i+1}) = +1$ for each $i \in [p]$ and $\sigma(v_0 v_{1}) = -1$. Thus, it follows that every edge contained in \(E(v_1,S')\) is a negative edge, and every edge contained in \(E(v_p,S')\) is a positive edge; otherwise, by Lemma \ref{G_5}(i), we can find a negative cycle of length $2k+1$. For each $i\in[2,p-1]$, by Lemma \ref{G_5}(i) and \ref{S_4}(i), we get $e(v_i,S')=0$. On the other hand, by Lemma \ref{S_4}(i), $E(\dot{G}[L']) \setminus E(C^{-})$ contains no positive edges and at most one negative edge, namely the edge $v_{1}v_{p}$; otherwise, we can find a negative cycle $C^{-}{'}$ such that $L' \setminus V(C^{-}{'}) \neq \emptyset$. Thus, $e(\dot{G}[L']) \le p$. Note that $N(v_1) \cap N(v_p) \cap S = \emptyset$. Recalling the choice of $C^{-}$, if $p \ge 4$, then
    \begin{equation}
        \begin{aligned}
        e(\dot{G})
        &= e(\dot{G}[L']) + e(\dot{G}[L',S']) + e(\dot{G}[S'])\\ 
        &\le p-1 + n-p + e(T_{r}(n-p))\\
        &\le n-1 + e(T_{r}(n-4))\\
        &\le \frac{r-1}{2r}(n^2 - \frac{6r-8}{r-1}n+9)\\
        &< e(T_{r}(n-2))+3,
        \end{aligned}
        \nonumber
    \end{equation}
    a contradiction. If $p \le 2$, then by Lemma \ref{G_5}(i), we can find a negative cycle of length $2k+1$, which is a contradiction. Now assume that $p=3$. If $k \ge 3$, then by Lemma \ref{G_5}(i), $\dot{G}$ contains a negative cycle of length $2k+1$, which is a contradiction. We assume that $k = 2$. If $N(v_1) \cap S' \setminus \{ v_0 \} \neq \emptyset$, we may assume $x \in N(v_1) \cap S' \setminus \{ v_0 \}$, and then $v_{p+1}$ is not adjacent to $x$; otherwise, $\dot{G}$ contains a negative cycle of length $2k+1$. Similarly, if $N(v_{p}) \cap S' \setminus \{ v_{p+1} \} \neq \emptyset$, we may assume $y \in N(v_{p}) \cap S' \setminus \{ v_{p+1} \}$, and then $v_{0}$ is not adjacent to $y$. Thus, 
    \begin{equation}
        \begin{aligned}
        e(\dot{G})
        &= e(\dot{G}[L']) + \left(e(\dot{G}[L',S']) + e(\dot{G}[S'])\right)\\ 
        &\le 2 + \left(2 + e(T_{r}(n-3)) - 1\right)\\
        &= e(T_{r}(n-3)) + 3\\
        &< e(T_{r}(n-2))+3,
        \end{aligned}
        \nonumber
    \end{equation}
    a contradiction. 
    
    {\flushleft {\it Subcase A.2.} $|V(C^{-}) \cap S'| = 1$.}

    Let $V(C^{-}) \cap S' = \{ v_0 \}$. Up to switching equivalence, we may assume that $\sigma(v_i v_{i+1}) = +1$ for each $i \in [p-1]$ and $\sigma(v_0 v_{1}) = -1$. Thus, every edge contained in \(E(v_1,S')\) is a negative edge and every edge contained in \(E(v_p,S')\) is a positive edge; otherwise, by Lemma \ref{G_5}(i), there exists a negative cycle of length $2k+1$ in $\dot{G}$. Similarly to Subcase $A.1$, for any $i \in [2,p-1]$, $e(v_i,S')=0$ and $e(\dot{G}[L']) \le p$. If $p \ge 2k+1$, then 
    \begin{equation}
        \begin{aligned}
        e(\dot{G})
        &= e(\dot{G}[L']) + e(\dot{G}[L',S']) + e(\dot{G}[S'])\\ 
        &\le p+ 2(1-\frac{1}{r}-7\varepsilon)n + e(T_{r}(n-p))\\
        & = \left(e(T_{r}(n-p))+p\right) + 2(1-\frac{1}{r}-7\varepsilon)n \\
        &\le e(T_{r}(n-5)) + 5  + 2(1-\frac{1}{r}-7\varepsilon)n \\
        &\le \frac{r-1}{2r}(n^2 - 6n+25)-14\varepsilon n +5\\
        &< e(T_{r}(n-2))+3,
        \end{aligned}
        \nonumber
    \end{equation}
    a contradiction. If $3 \le p \le 2k-2$, then $N(v_1) \cap S' \setminus \{ v_0 \} = \emptyset$ and $N(v_{p}) \cap S' \setminus \{ v_0 \} = \emptyset$; otherwise, by Lemma \ref{G_5}(i), $\dot{G}$ contains a negative cycle of length $2k+1$. Hence, we have 
    \begin{equation}
        \begin{aligned}
        e(\dot{G})
        &= e(\dot{G}[L']) + e(\dot{G}[L',S']) + e(\dot{G}[S'])\\ 
        &\le p+ 2 + e(T_{r}(n-p))\\
        &\le e(T_{r}(n-3))+5\\
        &< e(T_{r}(n-2))+3,
        \end{aligned}
        \nonumber
    \end{equation}
    a contradiction. If $p=2$, then $N(v_1)\cap S'\setminus\{v_0\}=\emptyset$ and $N(v_p)\cap S'\setminus\{v_0\}=\emptyset$; otherwise, by Lemma $\ref{G_5}$(i), there exists a negative cycle of length $2k+1$ in $\dot{G}$. Furthermore, since $p=2$, we have $e(\dot{G}[L']) = 1$. Thus, we get 
    \begin{equation}
        \begin{aligned}
        e(\dot{G})
        &= e(\dot{G}[L']) + e(\dot{G}[L',S']) + e(\dot{G}[S'])\\ 
        &\le 1 + 2 + e(T_{r}(n-2))\\
        &= e(T_{r}(n-2))+3,
        \end{aligned}
        \nonumber
    \end{equation}
    with equality if and only if $\dot{G}$ switching isomorphic to $C_{3}^{-} \cdot T_{r}(n-2)$. Since $p \neq 2k$, we assume below that $p=2k-1$. We assert that for any two distinct integers \(i,j\) with \(i,j \in [r]\), if \(u \in N(v_1) \cap S'_i\) and \(v \in N(v_p) \cap S'_j\), then \(u\) and \(v\) are non-adjacent; otherwise, this contradicts the fact that \(\dot{G}\) is \(C_{2k+1}^{-}\)-free. Without loss of generality, assume that \(v_0 \in S’_{1}\). If $e(v_{p},S’_{1})= 1$, for sufficiently large $n$, then we have 
    \begin{equation}
        \begin{aligned}
        e(\dot{G})
        &= e(\dot{G}[L']) + e(\dot{G}[L',S']) + e(\dot{G}[S'])\\ 
        &\le p + \left(e(u_1,S') + e(u_p,S') +e(\dot{G}[S']) \right)\\
        &\le 3 + \left( e(T_{r}(n-3)) + (1-\frac{1}{r}-7\varepsilon)n+1\right)\\
        &\le 4 + \left( \frac{r-1}{2r}(n^2-6n+9) + (1-\frac{1}{r}-7\varepsilon)n\right)\\
        &=  \frac{r-1}{2r}(n^2-4n+9) - 7\varepsilon n+4\\
        &< e(T_{r}(n-2))+3,
        \end{aligned}
        \nonumber
    \end{equation}
    a contradiction. If $e(v_{p},S’_{1}) > 1$, we have 

    \begin{equation}\label{eq::lg}
        \begin{aligned}
        e(\dot{G})
        &= e(\dot{G}[L']) + e(\dot{G}[L',S']) + e(\dot{G}[S'])\\ 
        &\le p+ \left(e(v_1,S') + e(v_p,S') +e(\dot{G}[S']) \right)\\
        &\le p + \left( 2|S'_{1}| + e(K_{|S'_{1}|,|S'_{2}|,\cdots,|S'_{r}|})\right)\\
        &= p + \left( 2|S'_{1}| + \frac{1}{2}( (n-p)^2 - \sum_{i=1}^r |S'_i|^2 )\right)\\
        &\le 3 +  \frac{1}{2}(n-3)^2+ 2|S'_{1}| - \frac{1}{2}\sum_{i=1}^r |S'_i|^2.
        \end{aligned}
    \end{equation}
    If $r\ge 4$, by the Lagrange multiplier method and straightforward calculations, we have
    \begin{equation}
        \begin{aligned}
        3 +  \frac{1}{2}(n-3)^2+ 2|S'_{1}|-\frac{1}{2} \sum_{i=1}^r |S'_i|^2
        &\le \frac{r(n-3)^2-(n-5)^2}{2r}+5\\
        &= \frac{r-1}{2r}(n-2)^2+\frac{6n-2rn + 9r - 21}{2r}+3\\
        &< e(T_r(n-2))+3, 
        \end{aligned}
        \nonumber
    \end{equation}
which is a contradiction. If $r=3$, applying the method of Lagrange multipliers analogously, we have 
    \begin{equation}
   3+\frac{1}{2}(n-3)^2+ 2|S'_{1}|-\frac{1}{2} \sum_{i=1}^3 |S'_i|^2 \le \left\lfloor \frac{n^2 - 4n + 16}{3} \right\rfloor.
        \nonumber
    \end{equation}
Moreover, the equality holds precisely in the following cases: if $n = 3k$, then
$(|S_1'|,|S_2'|,|S_3'|) = (k + 1,k - 2,k - 2)$ or $(k,k - 2,k - 1)$; if $n = 3k + 1$, then $(|S_1'|,|S_2'|,|S_3'|) = (k,k - 1,k - 1)$ or $(k + 1,k - 2,k - 1)$; if $n = 3k + 2$, then
$(|S_1'|,|S_2'|,|S_3'|) = (k + 1,k - 1,k - 1)$. Therefore, by inequality \eqref{eq::lg}, $e(\dot{G}) \le e(\dot{H})$, with equality if and only if $\dot{G}$ is switching isomorphic to $\dot{H}$. 

    {\flushleft {\it Case B.} For any negative cycle $C^{-}$, $V(C^{-}) \cap S' = \emptyset$.}

    Choose the shortest one among these negative cycles, and without loss of generality, denote it by $C^{-}{'} = (C', \sigma)$. Since $V(C^{-}{'}) \cap S' = \emptyset$, we may order the vertices of $C'$ as $v_1 v_2 \cdots v_p v_1$. Thus, there exists some \(i \in [p]\) such that \(e(v_i,S') \neq 0\). Without loss of generality, assume there exists a vertex \(v_0 \in S'\) adjacent to \(v_1\). Up to switching equivalence, we may assume that $\sigma(v_i v_{i+1}) = +1$ for each integer $i$ with $0 \le i \le p-1$ and $\sigma(v_{p} v_{1}) = -1$. Since $V(C^{-}) \cap S' = \emptyset$ for any negative cycle $C^{-}$, there is no positive edge in $E(L'\setminus \{v_1 \}, S')$. Furthermore, for each $i\in[2,p]$, we deduce from Lemma \ref{G_5}(i) and \ref{S_4}(i) that $e(v_i,S')=0$. Thus, we have 
    \begin{equation}
        \begin{aligned}
        e(\dot{G})
        &= e(\dot{G}[L']) + e(\dot{G}[L',S']) + e(\dot{G}[S'])\\ 
        &\le p + (1-\frac{1}{r}-7\varepsilon)n + e(T_{r}(n-p))\\
        &\le 3 + (1-\frac{1}{r}-7\varepsilon)n + e(T_{r}(n-3))\\
        &\le \frac{r-1}{2r}(n^2-6n+9)+ (1-\frac{1}{r}-7\varepsilon)n + 3\\
        &= \frac{r-1}{2r}(n^2-4n+9) -7\varepsilon n + 3\\
        &< e(T_{r}(n-2))+3,
        \end{aligned}
        \nonumber
    \end{equation}
    a contradiction. 

    This completes the proof. \hfill$\square$

    \noindent\textbf{The proof of Corollary \ref{cor::2}.} Let $\dot{G}$ be a signed graph satisfying the conditions of Corollary \ref{cor::2} with the maximum number of edges. Choose an integer $k\ge 3$ such that $2k+1\ge \ell$. Since $\dot{G}$ is $\mathcal{\dot{C}}_{\ge \ell}^{-}$-free, it follows that, for sufficiently large $n$, $\dot{G}$ contains no negative cycle of length $2k+1$. Note that $C_{3}^{-} \cdot T_{r}(n-2)$ is $\mathcal{\dot{C}}_{\ge \ell}^{-}$-free. By Theorem \ref{cor::1}, we have $e(\dot{G}) \le e(C_{3}^{-} \cdot T_{r}(n-2))$, with equality if and only if $\dot{G}$ is switching isomorphic to $C_{3}^{-} \cdot T_{r}(n-2)$.

\end{document}